\documentclass{siamart171218}
\usepackage{multirow}
\usepackage{braket,amsfonts,amsopn,bm}
\usepackage{graphicx,epstopdf}
\usepackage{array}
\usepackage[linesnumbered,ruled,algosection,vlined,algo2e,resetcount]{algorithm2e}

\usepackage{mathrsfs}
\usepackage{multirow}
\usepackage{amssymb,amsfonts,amsmath,latexsym,epsfig,euscript,color}
\usepackage[utf8]{inputenc}

\newcommand{\BL}[1]{\textcolor{black}{#1}}

\title{A new ParaDiag time-parallel time integration method}

\author{Martin J. Gander\thanks{Section de Math\`ematiques, University
    of Geneva, Switzerland ({\tt martin.gander@unige.ch})}
  \and Davide Palitta\thanks{Dipartimento di Matematica and AM$^2$, Alma Mater
    Studiorum - Universit\`a di Bologna, Piazza di Porta S. Donato, 5,
    I-40127 Bologna, Italy ({\tt davide.palitta@unibo.it})}}

\newcounter{mymac@matlab}
\newcommand{\MATLAB}{{\sc matlab}%
  \ifnum\value{mymac@matlab}<1%
  \textsuperscript{\textregistered}%
  \setcounter{mymac@matlab}{1}%
  \fi%
}

\begin{document}
\maketitle

\bibliographystyle{siam}

\begin{abstract}
Time-parallel time integration has received a lot of attention in the
high performance computing community over the past two decades. Indeed, it has been shown that parallel-in-time techniques have the potential to remedy one of the main computational drawbacks of parallel-in-space solvers. In particular, it is well-known that for large-scale evolution problems space parallelization
saturates long before all processing cores are effectively used on
today's large scale parallel computers. Among the many approaches for
time-parallel time integration, ParaDiag schemes have {\color{black}proven} to be
a very effective approach. In this framework, the time stepping matrix
or an approximation thereof is diagonalized by Fourier techniques, so that computations taking place at different time steps can be indeed carried out in parallel. We propose here a
new ParaDiag algorithm combining the Sherman-Morrison-Woodbury
formula and Krylov techniques. A panel of diverse numerical examples illustrates the potential of our new solver. In particular, we show that it performs very well
compared to different ParaDiag algorithms recently proposed in the literature.
\end{abstract}

\section{Introduction}

Time-parallel time integration is currently a very active field of
research within the high performance computing community. Research
interest {\color{black}was relaunched} over twenty years ago with the
introduction of the Parareal algorithm \cite{lions2001parareal}, which
is a two-level, non-intrusive method that allows existing codes to be
parallelized, and works well on parabolic problems; see,
e.g.,~\cite{gander2007analysis,gander2008nonlinear} for detailed
convergence analyses. Nowadays, many other different techniques for
time parallelization of evolution problems can be found on the market:
methods based on multiple shooting (like Parareal), methods based on
Domain Decomposition and Waveform relaxation, Multigrid type methods,
and even direct time-parallel solvers; see,
e.g.,~\cite{gander2015years} for a thorough review of such
schemes. The boundaries between these different solvers have become
less and less strict over the years, with tools designed for a
specific method being fully exploited in others.  The ParaDiag family
of time-parallel time integrators is a typical example of such
permeability. Originally, ParaDiag methods were designed as direct
time-parallel solvers \cite{MR2385067}. To overcome the
nondiagonalizability of the time stepping matrix which is a Jordan
block for a constant time step, it was first proposed to use different
time steps.  However, this trick can lead to very ill-conditioned
eigenvector matrices which can potentially pollute the entire solution
process, especially for fine time grids.  A thoughtful tradeoff
between having similar time steps for accuracy and different ones for
diagonalizability is thus crucial to make this first ParaDiag method
successful; see, e.g.,~\cite{Gander2016,Gander2019} for a detailed
analysis. Due to the {\color{black}practical} limitation this tradeoff
imposes on the number of time steps that can be parallelized, a recent
approach consists in considering an approximate problem where the time
stepping matrix is periodic\footnote{\BL{A anonymous reviewer pointed
    out that this corresponds to approximating a partial fraction
    expansion of the rational approximation for the matrix exponential
    defined by the numerical scheme (which is equivalent to a
    Weierstrass normal form \cite{frommer2021matrix}, here essentially
    the Jordan Canonical Form that was discovered independently
    \cite{hawkins1977weierstrass}), by imposing periodicity to obtain
    a normal form which is diagonal.}}. Thanks to this approximation,
the ParaDiag approach can be well combined with other time-parallel
methods as, e.g.,
Parareal~\cite{wu2018toward,gander2020diagonalization} or
MGRIT~\cite{wu2019acceleration}. See also \cite{gander2019convergence}
for time-periodic waveform relaxation for initial value problems and
\cite{gander2013analysis} for Parareal algorithms for truly
time-periodic problems.

A different approach consists in employing ParaDiag techniques within iterative methods. For instance, {\color{black} one approach approximates the time stepping matrix by a circulant matrix, with the latter being used to define} a
preconditioner for Krylov methods. A ParaDiag scheme can then be used for a more efficient application of the preconditioning operator; see, e.g., the iterative time
parallelization in~\cite{McDonald2018} for parabolic
problems, and~\cite{danieli2021all} for hyperbolic problems. Also for non-linear problems, ParaDiag algorithms
necessarily become iterative; see, e.g.,~ \cite{gander2017time}.

ParaDiag techniques have been developed for optimal control
problems~\cite{wu2020parallel} as well. In this setting, very good performance is obtained  by
using $\alpha$-circulant modifications of the time stepping matrix
\cite{liu2020fast,wu2021parallel}.  See also~\cite{gander2020paradiag} for an overview and implementation
details about these schemes.

In
contrast to the first attempts of time parallelization by
diagonalization using different time steps, these modern ParaDiag
methods are very successful in solving both parabolic and
hyperbolic evolution problems, and new ideas in this direction are
currently being developed; see, e.g.,~\cite{Kressner2022} where
interpolation and low-rank techniques are proposed and studied.

We present here a new ParaDiag algorithm for solving in a time-parallel fashion the evolution problem
\begin{equation}\label{eq:diff_prob}
 \begin{array}{rll}
         u_t&=&\mathfrak{L}(u)+f,\quad\text{in }\Omega\times(0,T],\\
         u&=&g,\quad\text{on }\partial\Omega,\\
         u(0)&=&u_0,\\
        \end{array}
\end{equation}
where the spatial domain $\Omega$ is such that $\Omega\subset\mathbb{R}^d$, $d=1,2,3$, and
$\mathfrak{L}$ is a linear differential operator involving only spatial
derivatives. If we discretize \eqref{eq:diff_prob} in space with a
finite element or finite difference method with $\bar n$ degrees of
freedom, and use a backward Euler scheme with $\ell$ time steps, the
\emph{all-at-once} discretization of~\eqref{eq:diff_prob} can be
written in matrix form as
\begin{equation}\label{eq:discrete_prob}
 (I+\tau K)U-U\Sigma_1^T=[\mathbf{u_0}+\tau\mathbf{f}_1,\ldots,\tau\mathbf{f}_\ell],
\end{equation}
where $K\in\mathbb{R}^{\bar n\times \bar n}$ is the stiffness matrix
stemming from the spatial discretization,
$\Sigma_1\in\mathbb{R}^{\ell\times\ell}$ is a zero matrix having ones
only on the first subdiagonal, the $j$-th column of
$U=[\mathbf{u}_1,\ldots,\mathbf{u}_\ell]\in\mathbb{R}^{\bar
  n\times \ell}$ represents the approximation to the solution $u$ at
time $t_j$, $j=1,\ldots,\ell$, $\tau=T/\ell$ is the time step, and
$\mathbf{u_0}$ and $\mathbf{f}_j$ gather the nodal values
of $u_0$ and $f(t_j)$ along with the boundary conditions. See~\cite{Pal2021}.

Recently, the matrix equation formulation~\eqref{eq:discrete_prob} has been used to design new solution procedures. In particular, low-rank solvers can be very successful in solving~\eqref{eq:discrete_prob} whenever the right-hand side $[\mathbf{u_0}+\tau\mathbf{f}_1,\ldots,\tau\mathbf{f}_\ell]$ has low rank; see, e.g.,
~\cite{Pal2021}.
On the other hand, the performance of such methods significantly worsens for right-hand sides with a sizable rank. In this paper, we {\color{black} address the performance issue when the right-hand side is possibly full rank.} In particular, we propose a new ParaDiag algorithm which is able to fully take advantage of the
circulant-plus-low-rank structure of $\Sigma_1$ so as to design an efficient
parallel-in-time algorithm for solving~\eqref{eq:discrete_prob}.
A significant advantage of this approach over low-rank space-time
schemes is {\color{black} that no} assumption on the
(numerical) rank of the right hand side
$[\mathbf{u_0}+\tau\mathbf{f}_1,\ldots,\tau\mathbf{f}_\ell]$ is needed in~\eqref{eq:discrete_prob}.

Here is an outline of the paper. In section~\ref{The novel algorithm} we show how to use the circulant-plus-low-rank structure of $\Sigma_1$ by combining~\eqref{eq:discrete_prob} with the matrix-oriented Sherman-Morrison-Woodbury formula. Such an approach sees as an intermediate step the solution of a linear system
whose coefficient matrix has a rather involved structure. An ad-hoc projection technique for the solution of this \emph{inner} linear system is proposed in section~\ref{GalerkinSection} where we also study some of the properties of the coefficient matrix in case of symmetric positive definite (SPD) stiffness matrices $K$. In section~\ref{alpha_acceleration} we propose a very successful variant of the ParaDiag scheme which makes use of $\alpha$-circulant matrices.  In section~\ref{Higher-order BDFs} we generalize our approach to the case of higher-order time discretization schemes. The potential of our new approach is illustrated in section~\ref{Numerical examples} where several numerical results are shown. In section~\ref{Conclusions} we draw our conclusions.

Throughout the paper we adopt the following notation. Capital letters $(A)$ denote matrices, bold, lower-case letters ($\mathbf{a}$) vectors, and plain, lower-case letters ($a$) scalars. $I_n=[\mathbf{e}_1,\ldots,\mathbf{e}_n]$ denotes the identity matrix of order $n$ and the subscript is omitted whenever the dimensions of $I$ are clear from the context; $\otimes$ is the Kronecker product whereas $\lambda_{\min}(A)$ and $\lambda_{\max}(A)$ denote the minimum and the maximum eigenvalue of $A$, respectively. Given a matrix $X=[\mathbf{x}_1,\ldots,\mathbf{x}_n]\in\mathbb{R}^{m\times n}$, $\text{vec}(X)\in\mathbb{R}^{mn}$ denotes the vector obtained by stacking the columns of $X$ on top of each other, namely $\text{vec}(X)=[\mathbf{x}_1^T,\ldots,\mathbf{x}_n^T]^T$.


\section{The new ParaDiag algorithm}\label{The novel algorithm}

Even though the discrete backward Euler operator $\Sigma_1$ cannot be diagonalized as it is a Jordan block, its {\color{black}circulant-plus-rank-one} structure can be exploited to design efficient solvers for~\eqref{eq:discrete_prob}. Indeed, we can write
\begin{equation}\label{eq:Sigma1_def}
 \Sigma_1=\begin{bmatrix}
            0 &  &  & \\
            1 & \ddots & & \\
            & \ddots & \ddots & \\
            & & 1 & 0 \\
           \end{bmatrix}=C_1-\mathbf{e}_1\mathbf{e}_\ell^T,\quad C_1=\begin{bmatrix}
            0 &  &  & 1\\
            1 & \ddots & & \\
            & \ddots & \ddots & \\
            & & 1 & 0 \\
           \end{bmatrix},
\end{equation}
where $C_1$ is a circulant matrix.

The relation above has been used in, e.g.,~\cite{McDonald2018} to derive preconditioning operators for the linear system counterpart of~\eqref{eq:discrete_prob}. In particular, such preconditioners were obtained by dropping the rank-1 term $\mathbf{e}_1\mathbf{e}_\ell^T$ in~\eqref{eq:Sigma1_def}.

In our setting, by inserting~\eqref{eq:Sigma1_def} into~\eqref{eq:discrete_prob}, we get
\begin{equation}\label{eq:discrete_prob_II}
  (I+\tau K)U-UC_1^T+U\mathbf{e}_\ell \mathbf{e}_1^T
    =[\mathbf{u_0}+\mathbf{f}_1,\ldots,\mathbf{f}_\ell].
\end{equation}
Since $C_1$ is a circulant matrix, it can be diagonalized by the Fast
Fourier Transform (FFT), namely
$\Pi_1={F}C_1{F}^{-1}=\text{diag}({F}C_1\mathbf{e}_1)=\text{diag}(\pi_1,\ldots,\pi_\ell)$,
where ${F}$ denotes the discrete Fourier matrix.

By postmultiplying~\eqref{eq:discrete_prob_II} by ${F}^T$ we get
\begin{equation}\label{eq:discrete_prob_III}
 (I_{\bar n}+\tau K)\widetilde U-\widetilde U\Pi_1+\widetilde U{F}^{-T}\mathbf{e}_\ell \mathbf{e}_1^T{F}^T=[\mathbf{u_0}+\tau\mathbf{f}_1,\ldots,\tau\mathbf{f}_\ell]{F}^T, \quad \widetilde U= U{F}^T,
\end{equation}
which can also be written in Kronecker form as
$$
\left(I_\ell\otimes(I_{\bar n}+\tau K)-\Pi_1\otimes I_{\bar n}+ {F}\mathbf{e}_1\mathbf{e}_\ell^T{F}^{-1}\otimes I_{\bar n}\right)\text{vec}(\widetilde U)=\text{vec}([\mathbf{u_0}+\tau\mathbf{f}_1,\ldots,\tau\mathbf{f}_\ell]{F}^T).
$$
If we look at the system matrix above, this can be viewed as the sum of two components. A main part
$P:=I_\ell\otimes(I_{\bar n}+\tau K)-\Pi_1\otimes I_{\bar n}$, and a
low-rank modification term $MN^T$ where $M:={F}\mathbf{e}_1\otimes I_{\bar n}$
and $N:={F}^{-T}\mathbf{e}_\ell\otimes I_{\bar n}$. The
Sherman-Morrison-Woodbury formula thus implies that
\begin{align}\label{eq:SMW_U}
 \text{vec}(\widetilde U)=&P^{-1}\text{vec}([\mathbf{u_0}+\tau\mathbf{f}_1,\ldots,\tau\mathbf{f}_\ell]{F}^T)\notag\\\
 &-P^{-1}M(I+N^TP^{-1}M)^{-1}N^TP^{-1}\text{vec}([\mathbf{u_0}+\tau\mathbf{f}_1,\ldots,\tau\mathbf{f}_\ell]{F}^T).
\end{align}
Since $P$ is block diagonal,
$$P=\begin{bmatrix}
     (1-\pi_1)I_{\bar n}+\tau K & & \\
     & \ddots & \\
     & & (1-\pi_\ell)I_{\bar n}+\tau K\\
    \end{bmatrix}\in\mathbb{R}^{\bar n\ell\times \bar n \ell},
$$
the action of its inverse can be efficiently computed in parallel. In particular, we can write
$$P^{-1}\text{vec}([\mathbf{u_0}+\tau\mathbf{f}_1,\ldots,\tau\mathbf{f}_\ell]{F}^T)=\text{vec}(L),$$
where
$$
\begin{aligned}
  &L:=[((1-\pi_1)I_{\bar n}+\tau K)^{-1}[\mathbf{u_0}+\tau\mathbf{f}_1,\ldots,\tau\mathbf{f}_\ell]{F}^T\mathbf{e}_1,\ldots,\\
    & \hspace*{10em}((1-\pi_\ell)I_{\bar n}+\tau K)^{-1}[\mathbf{u_0}+\tau\mathbf{f}_1,\ldots,\tau\mathbf{f}_\ell]{F}^T\mathbf{e}_\ell].
\end{aligned}
$$
By using the property of the Kronecker product,
 the last part in the second term on the right of \eqref{eq:SMW_U} is such that
$$N^TP^{-1}\text{vec}([\mathbf{u_0}+\tau\mathbf{f}_1,\ldots,\tau\mathbf{f}_\ell]{F}^T)= N^T\text{vec}(L)=L{F}^{-T}\mathbf{e}_\ell,$$
and the linear system with $I+N^TP^{-1}M$ we need to solve in~\eqref{eq:SMW_U} can thus be written as
\begin{equation}\label{eq:SMW_linearsystem}
  (I+N^TP^{-1}M)\mathbf{x}=\mathbf{b},\qquad \mathbf{b}:=L{F}^{-T}\mathbf{e}_\ell.
\end{equation}

A naive approach for solving~\eqref{eq:SMW_linearsystem} would be to first construct the coefficient matrix, and then apply one's favorite linear system solver. However,
we notice that we cannot explicitly compute the coefficient matrix
$I+N^TP^{-1}M$ as this {\color{black} requires applying $P^{-1}$ to all the $\bar
n$ columns of the matrix $M$, destroying potential parallel-in-time acceleration.}

We propose to use an ad-hoc projection scheme to solve~\eqref{eq:SMW_linearsystem}. To this end, we start by deriving an explicit expression for $N^TP^{-1}M$.  We
first notice that ${F}\mathbf{e}_1\in\mathbb{R}^\ell$ is the
vector of all ones, a property of the discrete Fourier transform. Therefore,
$M=[I_{\bar n},\ldots,I_{\bar n}]^T$. Moreover, if
$$
    {F}^{-T}\mathbf{e}_\ell=\begin{bmatrix}
                  \gamma_1\\
                  \vdots\\
                  \gamma_\ell\\
                 \end{bmatrix}\quad \text{ then }\quad
                  N=\begin{bmatrix}
                  \gamma_1I_{\bar n}\\
                  \vdots\\
                  \gamma_\ell I_{\bar n}\\
                 \end{bmatrix},              
$$
which implies that the term $N^TP^{-1}M$ we want to express explicitly
can be written as
\begin{align*}
N^TP^{-1}M=&\begin{bmatrix}
                  \gamma_1I_{\bar n}&
                  \cdots&
                  \gamma_\ell I_{\bar n}\\
                 \end{bmatrix}
\begin{bmatrix}
     ((1-\pi_1)I_{\bar n}+\tau K)^{-1} & & \\
     & \ddots & \\
     & & ((1-\pi_\ell)I_{\bar n}+\tau K)^{-1}\\
    \end{bmatrix}\begin{bmatrix}
                  I_{\bar n}\\
                  \vdots\\
                   I_{\bar n}\\
                 \end{bmatrix}\\
=&\sum_{i=1}^\ell \gamma_i
((1-\pi_i)I_{\bar n}+\tau K)^{-1}.
                 \end{align*}
Therefore, the linear system~\eqref{eq:SMW_linearsystem} can be
rewritten in the form
\begin{equation}\label{eq:SMW_linearsystemII}
  \underbrace{\left(I+\sum_{i=1}^\ell \gamma_i
  ((1-\pi_i)I_{\bar n}+\tau K)^{-1}\right)}_{=:J_\ell}\mathbf{x}=\mathbf{b}.
\end{equation}
We propose to apply a projection method for solving problem
\eqref{eq:SMW_linearsystemII} where the residual vector is imposed to be orthogonal to a suitable subspace. This is a very general machinery and
its effectiveness depends mainly on the approximation space one
uses. The details of our approach are given in
section~\ref{GalerkinSection}, and for the moment we simply denote by
$\mathbf{x}_m$ the computed approximation\footnote{The index $m$
  denotes the number of iterations performed by the projection method in
  Algorithm~\ref{alg:galerkin} to obtain the sought approximation; see
  section~\ref{GalerkinSection}.} to $\mathbf{x}$, solution
to~\eqref{eq:SMW_linearsystemII}.  Going back to~\eqref{eq:SMW_U}, we
use $\mathbf{x}_m$ to compute
$$
  W=P^{-1}M\mathbf{x}_m=P^{-1}\text{vec}(\mathbf{x}_m\mathbf{e}_1^T{F}^T)=P^{-1}\text{vec}(\mathbf{x}_m\mathbf{1}_\ell^T),
$$
where $\mathbf{1}_\ell\in\mathbb{R}^\ell$ denotes the vector of all
ones, and $P^{-1}$ can be applied column-wise in parallel once
again, since it is block diagonal. We thus get
$$
  W=[((1-\pi_1)I_{\bar n}+\tau K)^{-1}\mathbf{x}_m,\ldots,
  ((1-\pi_\ell)I_{\bar n}+\tau K)^{-1}\mathbf{x}_m].
$$
To conclude, the solution $U$ is then obtained by computing
$$
  U=(L-W){F}^{-T}.
$$
The pseudocode of our new ParaDiag method is given in Algorithm~\ref{alg:paradiag}.
\begin{algorithm}[t]
  \DontPrintSemicolon
  \SetKwInOut{Input}{input}\SetKwInOut{Output}{output}
  \Input{$K\in\mathbb{R}^{\bar n\times \bar n}$, $\mathbf{u}_0$, $\mathbf{f}_i\in\mathbb{R}^{\bar n}$, $i=1,\ldots,\ell$, $\ell\in\mathbb{N}$.}
  \Output{$U\in\mathbb{R}^{\bar n\times \ell}$ approximate solution to~\eqref{eq:discrete_prob}.}
  \BlankLine
  Compute $[\pi_1,\ldots,\pi_\ell]^T=
  {F}C_1\mathbf{e}_1$,  and $[\gamma_1,\ldots,\gamma_\ell]^T={F}^{-T}\mathbf{e}_\ell$\;

   \SetKwBlock{ParFor}
   {{\color{black} parfor} $i=1,\ldots,\ell$}{end}
  \ParFor{
    Set $L\mathbf{e}_i=((1-\pi_i)I_{\bar n}+\tau K)^{-1}[\mathbf{u}_0+\tau\mathbf{f}_1,\tau\mathbf{f}_2,\ldots,\tau\mathbf{f}_\ell]{F}^T\mathbf{e}_i$\;
    }
    
    Compute $\mathbf{b}=L{F}^{-T}\mathbf{e}_\ell$\;
    
    Compute $\mathbf{x}_m$ by applying Algorithm~\ref{alg:galerkin} to~\eqref{eq:SMW_linearsystemII}\;
    
  \SetKwBlock{ParFor}
   {{\color{black} parfor} $i=1,\ldots,\ell$}{end}
  \ParFor{
    Set $W\mathbf{e}_i=((1-\pi_i)I_{\bar n}+\tau K)^{-1}\mathbf{x}_m$\;
    }
    
    Set $U=(L-W){F}^{-T}$
    \caption{New ParaDiag Algorithm with Backward Euler}\label{alg:paradiag}
\end{algorithm}


\section{The projection method for the inner linear system}\label{GalerkinSection}

In this section we describe the solution
of~\eqref{eq:SMW_linearsystemII} by a suitable projection method.  In
particular, we compute a numerical approximation
$\mathbf{x}_m=V_m\mathbf{y}_m\approx \mathbf{x}$, where the
orthonormal columns of $V_m\in\mathbb{R}^{\bar n\times m}$ span a
suitable subspace $\mathcal{K}_m$, namely
$\mathcal{K}_m=\text{Range}(V_m)$, and the vector
$\mathbf{y}_m\in\mathbb{R}^m$ is computed by imposing an orthogonality
condition on the residual $\mathbf{r}_m=J_\ell
\mathbf{x}_m-\mathbf{b}$, i.e., $V_m^T\mathbf{r}_m=\mathbf{0}$.

Some computational aspects and the effectiveness of any projection
method strongly depend on the adopted approximation space
$\mathcal{K}_m$. {\color{black} Even though the structure of $J_\ell$ in~\eqref{eq:SMW_linearsystemII} is rather involved -- it is a linear combination of inverses of shifted and scaled $K$'s --
we propose using} the polynomial
Krylov subspace
\begin{equation}\label{eq:def_Krylov}
 \mathcal{K}_m={\color{black}\mathcal{K}_m(K,\mathbf{b})=}\text{span}\{\mathbf{b},K\mathbf{b},\ldots,K^{m-1}\mathbf{b}\},
\end{equation}
so that our projection method can be seen as a FOM-like scheme\footnote{Our projection scheme does not amount to the actual FOM method as the matrix used to define the subspace $\mathcal{K}_m$, namely $K$, does not coincide with the coefficient matrix of the linear system we want to solve, namely $J_\ell$.}.
However, other options {\color{black}such as} rational Krylov subspaces can be
used instead of~\eqref{eq:def_Krylov} as well. {\color{black} In our numerical experiments, using}~\eqref{eq:def_Krylov} already leads to remarkable performance, especially in terms of number of iterations;
see section~\ref{Numerical examples}.
  This means that using more sophisticated spaces, with a more expensive construction step, may not {\color{black}reap benefits.

  We would like to mention here that the FOM-like method we propose for solving~\eqref{eq:SMW_linearsystemII} can indeed be seen as a Krylov method for the numerical approximation of $f(K)\mathbf{b}$, with $f$ being the rational function $f(K)=\left(I+\sum_{i=1}^\ell \gamma_i
  ((1-\pi_i)I_{\bar n}+\tau K)^{-1}\right)^{-1}\mathbf{b}$. Further considerations about the connection between these two points of view will be made at the end of this section \BL{when $K$ is SPD}.
  }

The use of the approximation space~\eqref{eq:def_Krylov} allows us to compute the vector $\mathbf{y}_m$ and derive a relation for a cheap
computation of the residual norm $\|\mathbf{r}_m\|$, \BL{using
the Arnoldi relation arising from the construction of the Krylov subspace~{\color{black}\cite{Saad2003}},}
\begin{equation}\label{eq:polArnoldi}
  KV_m=V_mT_m+t_{m+1,m}\mathbf{v}_{m+1}\mathbf{e}_m^T,
\end{equation}
where $T_m=V_m^TKV_m$, $t_{m+1,m}$ stems from the orthogonalization
procedure, and $\mathbf{v}_{m+1}$ is the $(m+1)$-st basis vector. Since the relation \eqref{eq:polArnoldi} is
shift-invariant, we can shift $K$ using any eigenvalue
$\pi_i$ of the circulant matrix, $C_1$, giving
$$
((1-\pi_i)I_{\bar n}+\tau K)V_m=V_m((1-\pi_i)I_m+\tau T_m)+t_{m+1,m}\mathbf{v}_{m+1}\mathbf{e}_m^T,\quad\forall\;i=1,\ldots,\ell.
$$

Moreover, by premultiplying by $((1-\pi_i)I_{\bar n}+\tau K)^{-1}$,
postmultiplying by $((1-\pi_iI_m)+\tau T_m)^{-1}$, and moving some terms we get
\begin{align}\label{eq:shiftedArnoldi}
((1-\pi_i)I_{\bar n}+\tau K)^{-1}V_m=&V_m((1-\pi_i)I_m+\tau T_m)^{-1}\notag\\
&-t_{m+1,m}((1-\pi_i)I_{\bar n}+\tau K)^{-1}\mathbf{v}_{m+1}\mathbf{e}_m^T((1-\pi_i)I_m+\tau T_m)^{-1}.
\end{align}
Looking at the residual vector, if $\beta:=\|\mathbf{b}\|$,
$\mathbf{h}_i:=((1-\pi_i)I_{\bar n}+\tau K)^{-1}\mathbf{v}_{m+1}$, and
$S_i:=((1-\pi_i)I_{m}+\tau T_m)^{-1}$, we have
\begin{align*}
  \mathbf{r}_m=&J_\ell V_m\mathbf{y}_m-\mathbf{b}\\
=&V_m\left(\left(I_m+\sum_{i=1}^\ell \gamma_i
S_i-t_{m+1,m}V_m^T\sum_{i=1}^\ell \gamma_i\mathbf{h}_i\mathbf{e}_m^TS_i\right)\mathbf{y}_m-\beta \mathbf{e}_1\right)\\
&-t_{m+1,m}(I-V_mV_m^T)\sum_{i=1}^\ell \gamma_i\mathbf{h}_i\mathbf{e}_m^TS_i\mathbf{y}_m.\\
 \end{align*}
Imposing the Galerkin condition $V_m^T\mathbf{r}_m=0$ is thus equivalent to
computing $\mathbf{y}_m$ as the solution of the $m\times m$ linear system
\begin{equation}\label{eq:compute_ym}
  \left(I_m+\sum_{i=1}^\ell \gamma_i\left(I_m-t_{m+1,m}V_m^T
    \mathbf{h}_i\mathbf{e}_m^T\right)S_i\right)\mathbf{y}_m=\beta \mathbf{e}_1.
\end{equation}
Then the residual norm is such that
\begin{equation}\label{eq:compute_resnorm}
  \|\mathbf{r}_m\|=|t_{m+1,m}|\cdot \left\|(I-V_mV_m^T)
    \sum_{i=1}^\ell \gamma_i\mathbf{h}_i\mathbf{e}_m^TS_i\mathbf{y}_m\right\|.  
\end{equation}
Our FOM-like method for~\eqref{eq:SMW_linearsystemII} is summarized in
Algorithm~\ref{alg:galerkin}.

\begin{algorithm}[t]
\setcounter{AlgoLine}{0}

  \DontPrintSemicolon
  \SetKwInOut{Input}{input}\SetKwInOut{Output}{output}
  \Input{$K\in\mathbb{R}^{\bar n\times \bar n}$, $\mathbf{b}\in\mathbb{R}^{\bar n}$, $\pi_i$,$\gamma_i\in\mathbb{C}$, $i=1,\ldots,\ell$, $\text{{\rm maxit}}\in\mathbb{N}$, $\epsilon>0$, ${\color{black}q}\geq 1$.}
  \Output{$\mathbf{x}_{m}\in\mathbb{R}^{\bar n}$ approximate solution to~\eqref{eq:SMW_linearsystemII}.}
  \BlankLine
  Set $\beta=\|\mathbf{b}\|$, $\mathbf{v}_1=\mathbf{b}/\beta$, $V_1=\mathbf{v}_1$, $m=0$, $\|\mathbf{r}\|=1$\;
  
  \While{$\|\mathbf{r}\|>\epsilon\cdot\beta$ {\rm \textbf{and}} $m\leq \text{{\rm maxit}}$}{
    Set $m=m+1$\;

    Compute $\widetilde{\mathbf{v}}=K\mathbf{v}_m$\;
    
    \For{$k=1,\ldots,m$}{
    Compute $t_{k,m}=\mathbf{v}_k^T\widetilde{\mathbf{v}}$\;
    
    Set $\widetilde{\mathbf{v}}=\widetilde{\mathbf{v}}-t_{k,m}\mathbf{v}_k$\;
    }
    Set $t_{m+1,m}=\|\widetilde{\mathbf{v}}\|$, $\mathbf{v}_{m+1}=\widetilde{\mathbf{v}}/t_{m+1,m}$, and $V_{m+1}=[V_m,\mathbf{v}_{m+1}]$\;
    
   \If{$\mathtt{mod}(m,{\color{black}q})=0$}{
    \SetKwBlock{ParFor}
   {{\color{black} parfor} $i=1,\ldots,\ell$}{end}
  \ParFor{
      Set $\mathbf{h}_i=((1-\pi_i)I_{\bar n}+\tau K)^{-1}\mathbf{v}_{m+1}$ and
      $S_i=((1-\pi_i)I_m+\tau T_m)^{-1}$\;
    }
    Solve $(I_m+\sum_{i=1}^\ell\gamma_i(I_m -t_{m+1,j}V_m^T\mathbf{h}_i\mathbf{e}_j^T)S_i)\mathbf{y}_m=\beta \mathbf{e}_1$\;\label{alg:line_ym}

    Compute $\|\mathbf{r}\|=|t_{m+1,m}|\cdot\|(I-V_mV_m^T)\sum_{i=1}^\ell \mathbf{h}_i\mathbf{e}_m^TS_i\mathbf{y}_m\|$\;

    }

    }
    
    Set $\mathbf{x}_m=V_m\mathbf{y}_m$\;\label{alg:line_end}
  \caption{FOM-like method for~\eqref{eq:SMW_linearsystemII}\label{alg:galerkin}}
\end{algorithm}

The most computationally demanding step of this Krylov method is the
residual norm computation. In particular, the computation of the
vectors $\mathbf{h}_i$ requires the parallel solution of the linear systems
with $(1-\pi_i)I_{\bar n}+\tau K$ for all $i=1,\ldots,\ell$. In order
to reduce the computational cost, we may want to
solve~\eqref{eq:compute_ym} and compute~\eqref{eq:compute_resnorm}
only every ${\color{black}q}\geq 1$ iterations, {\color{black}namely the residual norm gets \emph{frozen} for ${\color{black}q}$ iterations.
In the worst case scenario, this procedure leads to computing a slightly larger subspace than what would have been necessary by checking~\eqref{eq:compute_resnorm}
at each iteration. On the other hand,}
the overall number of
parallel-in-time loops performed by Algorithm~\ref{alg:paradiag} becomes $m/{\color{black}q}+2$.

{\color{black}Another option could be using an \emph{inner} Krylov method to compute  the $ \mathbf{h}_i$'s. Since $((1-\pi_i)I_{\bar n}+\tau K)\mathbf{h}_i=\mathbf{v}_{m+1}$, by exploiting the fact the the right-hand side $\mathbf{v}_{m+1}$ does not depend on the shift index $i$ along with the shift-invariance property of the Krylov subspace, one may want to construct $\mathcal{K}_t(K,\mathbf{v}_{m+1})$ and employ well-established Krylov routines for shifted linear systems; see, e.g.~\cite{Frommer1998,Simoncini2003}. On the other hand, this procedure would compute only
approximations $ \mathbf{\widetilde h}_i\approx \mathbf{h}_i$ in general. The impact of such approximations on the vector $\mathbf{y}_m$ and, ultimately, on $\mathbf{x}_m$ may be tricky to assess. Moreover,
   an underlying assumption of this paper is that we are able to solve $\ell$ shifted linear systems with $K$ by using parallelization.

}
 

We now consider in more detail the special case where the stiffness
matrix $K$ is SPD, and we show that
our algorithm for~\eqref{eq:SMW_linearsystemII} can largely take
advantage of this structure. {\color{black} See the Appendix for the proof of the following result.}
\begin{theorem}\label{th:alg_galerkin}
  Let the spatial stiffness matrix $K$ be symmetric positive
  definite. Then the coefficient matrix $J_\ell:=I+\sum_{i=1}^\ell
  \gamma_i ((1-\pi_i)I_{\bar n}+\tau K)^{-1}$
  in~\eqref{eq:SMW_linearsystemII} is Hermitian positive definite for
  any $\ell\geq 1$.  Moreover,
  \begin{equation}\label{eq:ThAlg} 
    \kappa(J_\ell)\leq 1+\frac{1}{\tau\lambda_{\min}(K)}.
  \end{equation}
\end{theorem}
Recalling that Algorithm~\ref{alg:galerkin} imposes a Galerkin
condition on the Krylov subspace $\mathcal{K}_m$, we point out that
Theorem~\ref{th:alg_galerkin} implies that the solution $\mathbf{x}_m$
provided by Algorithm~\ref{alg:galerkin} minimizes the error in the
$J_\ell$-energy norm over $\mathcal{K}_m$ whenever $K$ is SPD.

It is well-known that the convergence rate of FOM-like methods for
symmetric positive definite problems is related to the condition
number of the coefficient matrix\footnote{Though better insight is provided by the entire
  eigenvalue distribution.}, see, e.g., \cite{LieStr13}. The
bound~\eqref{eq:ThAlg} displays the interplay between the spatial and
time discretization and how they contribute to the convergence of
Algorithm~\ref{alg:galerkin}. In particular, if $\lambda_{\min}(K)$ is
far from zero and the time grid is rather coarse, i.e. $\tau$ is
large, we expect Algorithm~\ref{alg:galerkin} to converge fast. On the
other hand, we may need many iterations to reach the desired level of
accuracy for problems posed on very fine time grids with a small
$\lambda_{\min}(K)$, depending on the scaling of $K$, e.g., $1/h^2$ or $1/h$.

Since $K$ is SPD, $\lambda_{\min}(K)$ can be cheaply computed by,
e.g., the inverse power method\footnote{The implementation of such a method must make use of the properties of $K$ to be efficient. In particular, its symmetric positive definite nature and the possible sparsity coming from the adopted discretization in space must be taken into account.}. Therefore, it is easy to check the
magnitude of $\tau\lambda_{\min}(K)$. On the other hand, one may want
to select a time grid such that the latter value is big enough to
guarantee a well-conditioned $J_\ell$.

For $K$ SPD, the solution of~\eqref{eq:compute_ym} by
Algorithm~\ref{alg:galerkin} can be seen as the numerical evaluation
of the action of the rational matrix function $f(K)=(I+\sum_{i=1}^\ell
\gamma_i ((1-\pi_i)I_{\bar n}+\tau K)^{-1})^{-1}$ on the vector $\mathbf{b}$,
namely $f(K)\mathbf{b}$, by the Lanczos method. Therefore, one may want to
take advantage of this viewpoint to derive a-priori upper bounds on the 2-
and/or $J_\ell$-norm of the error; see, e.g.,
\cite{FrommerSimoncini2008,FrommerEtal2013}. With these bounds at
hand, it would be possible to predict how many iterations are indeed
sufficient to get the desired level of accuracy in
Algorithm~\ref{alg:galerkin}, thus avoiding the expensive residual
norm computation in~\eqref{eq:compute_resnorm}. In this scenario, a
single parallel-in-time loop is necessary within
Algorithm~\ref{alg:galerkin}. In particular, once a sufficiently
large Krylov subspace is constructed, this parallel-in-time loop is
involved in the definition of the coefficient matrix
in~\eqref{eq:compute_ym} and thus in the computation of the vector
$\mathbf{y}_m$.

\section{$\alpha$-acceleration}\label{alpha_acceleration}

In the context of preconditioning operators for all-at-once linear
systems stemming from~\eqref{eq:diff_prob}, $\alpha$-circulant
matrices have been largely used; see,
e.g.,~\cite{liu2020fast}. We now explore the impact of this
technique also on the scheme we proposed in the previous sections. In
particular, given $\alpha\in(0,1]$, we write
\begin{equation}\label{eq:alpha_circulant}
\Sigma_1=C_\alpha-\alpha \mathbf{e}_1\mathbf{e}_\ell^T, \quad C_\alpha=\begin{bmatrix}
            0 &  &  & \alpha\\
            1 & \ddots & & \\
            & \ddots & \ddots & \\
            & & 1 & 0 \\
           \end{bmatrix}\in\mathbb{R}^{\ell\times\ell}.
\end{equation}
Now $C_\alpha$ is an $\alpha$-circulant matrix, and it can be diagonalized
by the \emph{scaled} Fast Fourier Transform
$$C_\alpha=D_\alpha^{-1}{F}^{-1}\Pi_\alpha {F}D_\alpha,\qquad D_\alpha=\begin{bmatrix}
     1 & & & \\
     & \alpha^{1/\ell} & & \\
     & & \ddots & \\
     & & & \alpha^{(\ell-1)/\ell}\\                                                                                      \end{bmatrix},\quad \Pi_\alpha=\alpha^{1/\ell}\Pi_1,
$$
see, e.g.,~\cite{Noschese2012}.

It is well-known that the
eigenvector matrix ${F}D_\alpha$ can be very ill-conditioned
for small $\alpha$ and sizable values of $\ell$. Indeed,
$\kappa({F}D_\alpha)=\alpha^{-(\ell-1)/\ell}$. Such ill-conditioning
will be one of the major obstacles in using very small values of
$\alpha$.

The use of~\eqref{eq:alpha_circulant} does not lead to any particular difficulty in the solver we proposed in section~\ref{The novel algorithm}. By following the same exact steps as in section~\ref{The novel algorithm} and adopting the same notation, a direct computation shows that~\eqref{eq:SMW_U} translates into
\begin{align}\label{eq:SMW_U_alpha}
 \text{vec}(\widetilde U)=&P^{-1}\text{vec}([\mathbf{u_0}+\tau\mathbf{f}_1,\ldots,\tau\mathbf{f}_\ell]D_\alpha{F}^T)\notag\\\
 &-\alpha^{1/\ell} P^{-1}M(I+\alpha^{1/\ell} N^TP^{-1}M)^{-1}N^TP^{-1}\text{vec}([\mathbf{u_0}+\tau\mathbf{f}_1,\ldots,\tau\mathbf{f}_\ell]D_\alpha{F}^T).
\end{align}
Once $\widetilde U$ is computed, the original solution can be retrieved by computing $U=\widetilde U{F}^{-T}D_\alpha^{-1}$.

From~\eqref{eq:SMW_U_alpha}, it is clear how the use of $\alpha$-circulant matrices is equivalent to introducing a \emph{weight} in our setting. Indeed,
the scalar $\alpha\in(0,1]$ determines the contribution of the correction term
\begin{align*}
  \text{vec}(U_2)=&\alpha^{1/\ell}(D_\alpha^{-1}{F}^{-1}\otimes I) P^{-1}M(I+\alpha^{1/\ell} N^TP^{-1}M)^{-1}N^TP^{-1}\\
  &\text{vec}([\mathbf{u_0}+\tau\mathbf{f}_1,\ldots,\tau\mathbf{f}_\ell]D_\alpha{F}^T),
\end{align*}
to the point that, for sufficiently small $\alpha$, the solution can
be often well approximated even when neglecting this term; see
section~\ref{Numerical examples}.

The inner FOM-like method can benefit from the introduction of the
parameter $\alpha$ as well. Indeed, it holds that
\begin{equation}\label{eq:innersystem_alpha}
I+\alpha^{1/\ell} N^TP^{-1}M= I+\alpha^{1/\ell}\sum_{i=1}^\ell \gamma_i
((1-\pi_i)I_{\bar n}+\tau K)^{-1}.
\end{equation}
Therefore, we need to solve a linear system whose coefficient matrix
can be seen as a small perturbation of the identity for $\alpha\ll
1$. Using small values of $\alpha$ has thus the potential of
remarkably reducing the number of iterations performed by the
FOM-like method to attain a prescribed accuracy. This motivates the
name \emph{$\alpha$-acceleration} whenever $\alpha$-circulant
matrices of the form~\eqref{eq:alpha_circulant} are used in our
context.

For the sake of completeness we report in
Algorithm~\ref{alg:alpha_paradiag} the overall procedure when using
the $\alpha$-circulant matrix in~\eqref{eq:alpha_circulant}.  In
line~\ref{alg:ifline} we compute the residual norm provided by the
first term of the solution, namely
$\text{vec}(U_1)=(D_\alpha^{-1}{F}^{-1}\otimes I)
P^{-1}\text{vec}([\mathbf{u_0}+\tau\mathbf{f}_1,\ldots,\tau\mathbf{f}_\ell]D_\alpha{F}^T)$. If
this residual norm is sufficiently small, we stop the algorithm and
set $U=U_1$ thus saving a lot of computational effort.

\begin{algorithm}[t]
\setcounter{AlgoLine}{0}

  \DontPrintSemicolon
  \SetKwInOut{Input}{input}\SetKwInOut{Output}{output}
  \Input{$K\in\mathbb{R}^{\bar n\times \bar n}$, $\mathbf{u}_0$, $\mathbf{f}_i\in\mathbb{R}^{\bar n}$, $i=1,\ldots,\ell$, $\ell\in\mathbb{N}$, $\alpha\in(0,1]$, $\tau,\epsilon>0$.}
  \Output{$U\in\mathbb{R}^{\bar n\times \ell}$ approximate solution to~\eqref{eq:discrete_prob}.}
  \BlankLine
  Compute $[\pi_1,\ldots,\pi_\ell]^T=
  \alpha^{1/\ell}{F}C_1\mathbf{e}_1$, $[\gamma_1,\ldots,\gamma_\ell]^T={F}^{-T}\mathbf{e}_\ell$, $D_\alpha=\text{diag}(1,\alpha^{1/\ell},\ldots, \alpha^{(\ell-1)/\ell})$\;

  \SetKwBlock{ParFor}
   {{\color{black} parfor} $i=1,\ldots,\ell$}{end}
  \ParFor{
    Set $L\mathbf{e}_i=((1-\pi_i)I_{\bar n}+\tau K)^{-1}[\mathbf{u}_0+\tau\mathbf{f}_1,\tau\mathbf{f}_2,\ldots,\tau\mathbf{f}_\ell]D_\alpha{F}^T\mathbf{e}_i$\;
    }

    Set $U_1=L{F}^{-T}D_\alpha^{-1}$\;

    \If{$\|(I+\tau K)U_1-U_1\Sigma_1^T-[\mathbf{u_0}+\tau\mathbf{f}_1,\ldots,\tau\mathbf{f}_\ell]\|_F\leq \epsilon\cdot\|U_1\|_F$\label{alg:ifline}}{

    Set $U=U_1$ and \textbf{return}\;
    }
    Compute $\mathbf{b}=L{F}^{-T}\mathbf{e}_\ell$\;

    Compute $\mathbf{x}_m$ by applying a variant of
    Algorithm~\ref{alg:galerkin} to
    $$
      \left(I+\alpha^{1/\ell}\sum_{i=1}^\ell \gamma_i
      ((1-\pi_i)I_{\bar n}+\tau K)^{-1}\right)\mathbf{x}=\mathbf{b}
    $$\label{alg:linegalerkin}%

  \SetKwBlock{ParFor}
   {{\color{black} parfor} $i=1,\ldots,\ell$}{end}
  \ParFor{Set $W\mathbf{e}_i=((1-\pi_i)I_{\bar n}+\tau K)^{-1}\mathbf{x}_m$\;
    }

    Set $U=U_1-\alpha^{1/\ell}W{F}^{-T}D_\alpha^{-1}$
    \caption{ParaDiag and Backward Euler with $\alpha$-acceleration}\label{alg:alpha_paradiag}
\end{algorithm}

Even though the use of tiny values of $\alpha$ would be desirable,
this would lead to an extremely ill-conditioned eigenvector matrix
${F}D_\alpha$. For very small values of $\alpha$, {\color{black}this poor} conditioning
pollutes the computed solution $U$. In particular, it can be
observed that the error is proportional to
$\kappa({F}D_\alpha)$. This drawback is well-known also in
the case of the variable time-stepping procedure proposed
in~\cite{Gander2016,Gander2019}. Finding the value of $\alpha$
providing the best trade-off between acceleration and loss of accuracy
is not an easy task.  We refer to our numerical experiments to supply
some insight to this important issue.

\section{Higher-order BDFs}\label{Higher-order BDFs}

In the previous sections we assumed that the backward Euler method was used for time discretization so that $\Sigma_1$ was of the form~\eqref{eq:Sigma1_def}.
In this section we show how to generalize our methodology {\color{black} to higher-order \BL{BDF} approximations of time-dependent problems~\eqref{eq:diff_prob}.}

The backward Euler method belongs to the larger class of backward
differentiation formulas (BDFs) for time-dependent
problems~\eqref{eq:diff_prob}. In particular, the backward Euler
method is a BDF of order one, and BDFs of order $s>1$ are
well-established time integrator schemes as well.

The matrix formulation
of the discrete problem stemming from {\color{black} an all-at-once discretization of~\eqref{eq:diff_prob} using} a BDF of order $s$ can be written as
\begin{equation}
 \label{eq:discrete_probBDF}
 (I+\tau\beta K)U-U\Sigma_s^T=G,
\end{equation}
where
$$G=\left[\sum_{j=1}^s\alpha_j\mathbf{u}_{1-j}+\tau\beta\mathbf{f}_1,\sum_{j=2}^s\alpha_j\mathbf{u}_{2-j}+\tau\beta\mathbf{f}_2\ldots,\alpha_s\mathbf{u}_{0}+\tau\beta\mathbf{f}_s,\tau\beta\mathbf{f}_{s+1},\ldots\tau\beta\mathbf{f}_\ell\right],$$
and
\begin{equation}
 \Sigma_s=\begin{bmatrix}
     0 & & & &&& \\
     \alpha_1 & \ddots & & &&&\\
     \vdots & \ddots& & &&&\\
     \alpha_s &  & \ddots& \ddots&&&\\
     0 & \ddots & &\ddots&&&\\
     \vdots & &\ddots &&\ddots&\ddots&\\
     0 & &\cdots&\alpha_s&\cdots& \alpha_1 & 0\\
    \end{bmatrix},
\end{equation}
see, e.g., \cite{Pal2021}.

In the equation above, $\mathbf{u}_{1-s},\ldots, \mathbf{u}_{0}$
denote the $s$ initial values required by the adopted $s$-order
BDF. In Table~\ref{tabBDF} we report the scalars
$\alpha_j=\alpha_j(s)$ and $\beta=\beta(s)$ defining a BDF of order
$s$; see, e.g.,~\cite[Table 5.3]{Ascher1998}\footnote{Notice that in
  Table~\ref{tabBDF} we have changed the sign of the $\alpha_j$'s with
  respect to the values listed in~\cite[Table 5.3]{Ascher1998} in order to have a term of the form $-U\Sigma_s^T$ also in~\eqref{eq:discrete_probBDF} as it is done for equation~\eqref{eq:discrete_prob}.}.
\begin{table}[!t]
 \centering
 \caption{
 BDF coefficients for $s\leq 6$.\label{tabBDF}}
\begin{tabular}{r r r r r rrr }
 $s$ & $\beta$  & $\alpha_1$ & $\alpha_2$ & $\alpha_3$ & $\alpha_4$ & $\alpha_5$ & $\alpha_6$\\
\hline
 1 & 1 & 1 & & & \\
 2 & 2/3 & 4/3 & -1/3 & & \\
 3 & 6/11 & 18/11 & -9/11 & 2/11 & \\
 4 & 12/25 & 48/25 &-36/25 & 16/25 & -3/25 & \\
 5 & 60/137 & 300/137 & -300/137 & 200/137 & -75/137 & 12/137 \\
 6 &60/147 & 360/147 &-450/147 & 400/147 &-225/147 & 72/147 &-10/147\\
  \end{tabular}
 \end{table}
It has been shown in, e.g.,~\cite[Section 5.2.3]{Ascher1998} that any BDF of order $s>6$ is unstable. We thus restrict ourselves to $s\leq 6$.

For a general BDF of order $s$, the time discrete operator $\Sigma_s$ can be written as a circulant plus a  matrix of rank $s$, namely
$$\Sigma_s=C_s-[\mathbf{e}_1,\ldots,\mathbf{e}_s]\boldsymbol{\alpha}_s[\mathbf{e}_{\ell-s+1},\ldots,\mathbf{e}_\ell]^T,\ \boldsymbol{\alpha}_s=\begin{bmatrix}
 \alpha_s&\ldots & &\ldots & \alpha_1\\
& \alpha_s&\ldots & \ldots & \alpha_1\\
& & \ddots& & \vdots\\
& & & \alpha_s& \alpha_{s-1}\\
& & & &\alpha_s\\
\end{bmatrix}\in\mathbb{R}^{s\times s}.
$$
This means that the same strategy presented in section~\ref{The
  novel algorithm} can be applied by setting
in~\eqref{eq:SMW_U} $P:=I_\ell\otimes (I_{\bar n}+\tau\beta
K)-\Pi_s\otimes I_{\bar n}$, $C_s:={F}\Pi_s{F}^{-1}$, $\Pi_s=\text{diag}(FC_s\mathbf{e}_1)$,
$M:={F}[\mathbf{e}_1,\ldots,\mathbf{e}_s]\otimes I_{\bar n}$, and
$N:={F}^{-T}[\mathbf{e}_{\ell-s+1},\ldots,\mathbf{e}_\ell]\boldsymbol{\alpha}_s^T$. The
main difference lies in the linear
system~\eqref{eq:SMW_linearsystem}. In this setting, the right-hand
side, simply denoted by $\mathbf{b}$ in~\eqref{eq:SMW_linearsystem}, is given
by $\mathbf{b}^{(s)}=\text{vec}(B^{(s)})$ where
{\small
$$B^{(s)}=[((1-\pi_1)I_{\bar n}+\tau\beta K)^{-1}G{F}^T\mathbf{e}_1,\ldots,((1-\pi_\ell)I_{\bar n}+\tau\beta K)^{-1}G{F}^T\mathbf{e}_\ell]{F}^{-T}[\mathbf{e}_{\ell-s+1},\ldots,\mathbf{e}_\ell]\boldsymbol{\alpha}_s^T.
$$}
Moreover, the coefficient matrix $I+N^TP^{-1}M\in\mathbb{R}^{s\bar
  n\times s\bar n}$ can be seen as an $s\times s$ block matrix with
$\bar n\times \bar n$ blocks; see, e.g., \cite[Section 5.2]{Pal2021}.
If
$${F}\mathbf{e}_i=\begin{bmatrix}
                  \theta_1^{(i)}\\
                  \vdots\\
                  \theta_\ell^{(i)}\\
                 \end{bmatrix},%
                 \quad 
{F}^{-T}[\mathbf{e}_{\ell-s+1},\ldots,\mathbf{e}_\ell]\boldsymbol{\alpha}_s\mathbf{e}_i=
\begin{bmatrix}
                  \gamma_1^{(i)}\\
                  \vdots\\
                  \gamma_\ell^{(i)}\\
                 \end{bmatrix},\quad 
                              i=1,\ldots,s,
   $$
  so that 
$$M=\begin{bmatrix}
      \theta_1^{(1)}I_{\bar n}&\cdots& \theta_1^{(s)}I_{\bar n}\\
                  \vdots&& \vdots\\
                  \theta_\ell^{(1)}I_{\bar n}&\cdots&\theta_\ell^{(s)}I_{\bar n}\\
    \end{bmatrix},\quad N=\begin{bmatrix}
      \gamma_1^{(1)}I_{\bar n}&\cdots& \gamma_1^{(s)}I_{\bar n}\\
                  \vdots&& \vdots\\
                  \gamma_\ell^{(1)}I_{\bar n}&\cdots&\gamma_\ell^{(s)}I_{\bar n}\\
    \end{bmatrix},
$$
then we have
{\small
$$
 N^TP^{-1}M =  \begin{bmatrix}
   \displaystyle\sum_{i=1}^\ell \gamma_i^{(1)}\theta_i^{(1)}
((1-\pi_i)I_{\bar n}+\tau\beta K)^{-1}

&\cdots& \displaystyle\sum_{i=1}^\ell \gamma_i^{(1)}\theta_i^{(s)}
((1-\pi_i)I_{\bar n}+\tau\beta K)^{-1}
\\
                  \vdots&& \vdots\\
   \displaystyle         \sum_{i=1}^\ell \gamma_i^{(s)}\theta_i^{(1)}
((1-\pi_i)I_{\bar n}+\tau\beta K)^{-1}
&\cdots&\displaystyle\sum_{i=1}^\ell \gamma_i^{(s)}\theta_i^{(s)}
((1-\pi_i)I_{\bar n}+\tau\beta K)^{-1}
\\
    \end{bmatrix}.
    $$}
Denoting the matrix $I+N^TP^{-1}M$ by $J_\ell^{(s)}$,
we thus need to solve the linear system
\begin{equation}\label{eq:linearsystem_s}
J_\ell^{(s)}\mathbf{x}=\mathbf{b}^{(s)}.
\end{equation}
To this end, we adopt a projection scheme similar to the one presented
in section~\ref{GalerkinSection}. Instead of~\eqref{eq:def_Krylov}, we
use the block Krylov subspace
\begin{equation}\label{eq:def_blockKrylov}
 \mathcal{K}_m^{\square}=\text{range}\{[B^{(s)},KB^{(s)},\ldots,K^{m-1}B^{(s)}]\}.
\end{equation}
If $V_m=[\mathcal{V}_1,\ldots,\mathcal{V}_m]\in\mathbb{R}^{n\times sm}$, $\mathcal{V}_i\in\mathbb{R}^{n\times s}$ for all $i=1,\ldots,m$, has orthonormal columns and is such that $\text{range}(V_m)=\mathcal{K}_m^{\square}$, then the block counterpart of the Arnoldi relation~\eqref{eq:polArnoldi} holds. In particular,
 \begin{equation}\label{eq:blockpolArnoldi}
 KV_m=V_mT_m+\mathcal{V}_{m+1}t_{m+1,m}(\mathbf{e}_m^T\otimes I_s),
 \end{equation}
where $T_m=V_m^TKV_m\in\mathbb{R}^{sm\times sm}$, and
$t_{m+1,m}\in\mathbb{R}^{s\times s}$; see, e.g.,
\cite{Gutknecht2006,Kirk2020} for further details on block Krylov
methods and efficient schemes for the computation of the basis $V_m$.

In view of~\eqref{eq:blockpolArnoldi} we can derive the block
counterparts of~\eqref{eq:shiftedArnoldi}. In particular, by using the
same arguments of section~\ref{GalerkinSection}, we have
\begin{equation}\label{eq:blockshiftedArnoldi}
 \begin{aligned}
&((1-\pi_i)I_{\bar n}+\tau\beta K)^{-1}V_m=V_m((1-\pi_i)I_m+\tau\beta T_m)^{-1}\\
&\qquad-((1-\pi_i)I_{\bar n}+\tau\beta K)^{-1}\mathcal V_{m+1}t_{m+1,m}(\mathbf{e}_m^T\otimes I_s)((1-\pi_i)I_m+\tau\beta T_m)^{-1}.
 \end{aligned}
\end{equation}
We propose to solve~\eqref{eq:linearsystem_s} by imposing a Galerkin
condition with respect to the space spanned by $I_s\otimes V_m$. In
particular, if $\mathbf{x}_m=(I_s\otimes V_m)\mathbf{y}_m$, we want
the residual vector
$\mathbf{r}_m=J_\ell^{(s)}\mathbf{x}_m-\mathbf{b}^{(s)}$ to be such
that $(I_s\otimes V_m^T)\mathbf{r}_m=0$.
\begin{proposition}
  If $B^{(s)}=V_1\beta$, $\beta\in\mathbb{R}^{s\times s}$,
  $H_i=((1-\pi_i)I_{\bar n}+\tau\beta K)^{-1}\mathcal V_{m+1}$, and
  $S_i=((1-\pi_i)I_m+\tau\beta T_m)^{-1}$, then the residual vector
  $\mathbf{r}_m=J_\ell^{(s)}\mathbf{x}_m-\mathbf{b}^{(s)}=J_\ell^{(s)}(I_s\otimes
  V_m)\mathbf{y}_m-\mathbf{b}^{(s)}$ can be written as
  \begin{align*}
    \mathbf{r}_m=&(I_s\otimes V_m)\left(\left(I_{sm}+
    \mathcal{T}_m\right)\mathbf{y}_m-(\mathbf{e}_1\otimes
    I_s)\beta\right)-(I-I_s\otimes V_mV_m^T)\mathcal{S}_m\mathbf{y}_m,
  \end{align*}
  where $\mathcal{T}_m,\mathcal{S}_m\in\mathbb{R}^{ms\times ms}$ are 
  $s\times s$ block matrices whose $(k,h)$ blocks
  $(\mathcal{T}_m)_{k,h}$, $(\mathcal{S}_m)_{k,h}$ are given by
  $$
    (\mathcal{T}_m)_{k,h}=\sum_{i=1}^\ell \gamma_i^{(k)}\theta_i^{(h)}
  S_i-V_m^T\sum_{i=1}^\ell \gamma_i^{(k)}\theta_i^{(h)}H_it_{m+1,m}(\mathbf{e}_m^T\otimes I_s)S_i,
  $$
  $$
  (\mathcal{S}_m)_{k,h}=\sum_{i=1}^\ell \gamma_i^{(k)}\theta_i^{(h)}H_it_{m+1,m}(\mathbf{e}_m^T\otimes I_s)S_i.
  $$
  Moreover, imposing the Galerkin condition
  $\mathbf{r}_m\perp\mathcal{K}_m^{\square}$ is equivalent to
  computing $\mathbf{y}_m$ as the solution of the linear system
  $$
  \left(I_{sm}+\mathcal{T}_m\right)\mathbf{y}_m
    =(\mathbf{e}_1\otimes I_s)\beta,
  $$
  and
  $$
    \|\mathbf{r}_m\|=\|(I-I_s\otimes V_mV_m^T)\mathcal{S}_m\mathbf{y}_m\|.
  $$
\end{proposition}
\begin{proof}
  The results can be shown by applying the exact same arguments of
  section~\ref{GalerkinSection} \emph{block-wise} noticing that
  $$
    I_s\otimes V_m=\begin{bmatrix}
                   V_m & & \\
                   & \ddots & \\
                   & & V_m\\
                  \end{bmatrix}.
  $$
\end{proof}


\section{Numerical examples}\label{Numerical examples}

In this section we show several numerical examples to illustrate different aspects of our
new ParaDiag algorithm. We would like to mention that in the following we refrain from reporting any running times achieved by the routines we test. Indeed, all the experiments are carried out on a
simple laptop, so that reporting running times would be only informative
in a relative sense. Moreover, the application of $P$ is still
performed sequentially and no real parallel-in-time implementation is
adopted so far. On the other hand, we report the number of parallel-in-time loops that would be performed by the tested schemes as we believe this can be a fair measure of the computational cost.

Unless stated otherwise, the threshold used in
Algorithm~\ref{alg:galerkin} is $\epsilon=10^{-8}$. Moreover, we check
the residual norm in Algorithm~\ref{alg:galerkin} at each iteration,
i.e. we set ${\color{black}q}=1$.

\subsection{Heat Equation}

We consider the heat equation
\begin{equation}\label{eq:heat_eq.Ex1}
 \left\{
\begin{array}{rlll}
         u_t&=&\Delta u+f,& \quad \text{in }\Omega\times (0,1],\\
         u&=&0,&\quad\text{on }\partial\Omega,\\
         u(x,0)&=&u_0,&
        \end{array}\right.
\end{equation}
where $\Omega=(0,1)^2$. If the $2$-dimensional Laplace
operator is discretized using second order centered finite differences
with $n$ nodes in each direction and $h=1/(n+1)$, the stiffness matrix
$K\in\mathbb{R}^{\bar n\times \bar n}$, $\bar n=n^2$, can be written
as
$${\color{black}K=I_n\otimes T+T\otimes I_n},\quad T=\frac{1}{h^2}\begin{pmatrix}
 2 & -1 & & \\
 -1 & \ddots & \ddots &\\
 & \ddots & \ddots &-1 \\
 &&-1& 2\\                                                               \end{pmatrix}\in\mathbb{R}^{n\times n}.
$$
Since the eigenvalues of the SPD matrix $T$ are known in closed-form, by using the properties of the Kronecker product,
it is easy to show that
$$
  \lambda_{\min}(K)=\frac{8}{h^2}\sin\left(\frac{\pi}{2(n+1)}\right)^2.
$$
By defining $\tau=1/\ell$ for a given number of time steps $\ell\geq
1$, we can compute the upper bound~\eqref{eq:ThAlg}. In
Table~\ref{tab1} we show this upper bound for common values of $\bar
n$ and $\ell$.
\begin{table}
 \centering
 \caption{Upper bound {\color{black} of the coefficient matrix, $J_\ell$,}~\eqref{eq:ThAlg} for the
   $2$-dimensional heat equation~\eqref{eq:heat_eq.Ex1}. Different values of the number of spatial degrees of freedom
   $\bar n$ and time steps $\ell$ are used in the
   discretization.\label{tab1}}
\begin{tabular}{r r r }
 $\bar n$ & $\ell$  & $1+\frac{1}{\tau\lambda_{\min}(K)}$ \\
\hline
 
\multirow{3}{*}{65536}
  & 256 & 13.969 \\
 & 512 & 26.938 \\
  & 1024 & 52.877 \\
\hline
\multirow{3}{*}{262144}
  & 256 & 13.969 \\
 & 512 & 26.938 \\
  & 1024 & 52.876 \\
\hline
\multirow{3}{*}{1048576}
  & 256 & 13.969 \\
 & 512 & 26.938 \\
  & 1024 & 52.876 \\
  \end{tabular}
 \end{table}
From the results in Table~\ref{tab1} we first notice that while the
value of the upper bound~\eqref{eq:ThAlg} linearly grows with $\ell$,
an increment in the number of spatial degrees of freedom $\bar n$ does
not have a significant impact since $\lambda_{\min}(K)$ only
moderately increases with $n$. Moreover, $1+1/(\tau\lambda_{\min}(K))$
is $\mathcal{O}(10)$ for any value of $\bar n$ and $\ell$ we
tested. This implies that the matrix $J_\ell$ is always very
well-conditioned regardless of the source term $f$ and the initial
condition $u_0$ in~\eqref{eq:heat_eq.Ex1}.

We tested Algorithm~\ref{alg:paradiag} on an instance of the heat
equation~\eqref{eq:heat_eq.Ex1}. In particular, we consider
\cite[Example 6.1]{McDonald2018} where $f=0$ and
$u_0=xy(x-1)(1-y)$.
{\color{black}
 For all the values of $\bar n$ and $\ell$ listed in Table~\ref{tab1}, a single iteration of Algorithm~\ref{alg:galerkin} already
results in a relative residual norm smaller than $\epsilon = 10^{-8}$.} Indeed, a more careful inspection of the
eigenvalue distribution of $J_\ell$ shows that the latter is a very small
perturbation of the identity. In particular,
$$
  \lambda_{\min}(J_\ell)=1+\sum_{i=1}^\ell\frac{\gamma_i}{1-\pi_i+\tau\lambda_{\max}(K)},\quad \lambda_{\max}(J_\ell)=1+\sum_{i=1}^\ell\frac{\gamma_i}{1-\pi_i+\tau\lambda_{\min}(K)},
$$
and for all the values of $\bar n$ and $\ell$ listed in
Table~\ref{tab1} we have
$[\lambda_{\min}(J_\ell),\lambda_{\max}(J_\ell)]\subset [1,1.004]$.

\subsection{Advection-Diffusion Problem}
We now consider the time-dependent advection-diffusion equation
\begin{equation}\label{Ex.3_eq}
 \left\{\begin{array}{rlll}
         u_t-\nu\Delta u + \mathbf{w}\cdot\nabla u&=&0,&  \text{in }\Omega\times (0,1],\;\Omega\colon=(0,1)^2,\\
         u&=&g(x,y),&\text{on }\partial\Omega,\\
         u_0=u(x,y,0)&=&0&\text{otherwise,}
        \end{array}\right.
\end{equation}
where $\nu>0$, $\mathbf{w}=(2y(1-x^2),-2x(1-y^2))$ and
$g(1,y)=g(x,0)=g(x,1)=0$ while $g(0,y)=1$;
see, e.g.,~\cite{McDonald2018}.
To obtain the stiffness matrix
$K\in\mathbb{R}^{\bar n\times \bar n}$, we used again centered finite
differences, and for the time integration backward Euler. The initial vector $\mathbf{u}_0$ is set to be zero everywhere except the boundaries, where it satisfies the boundary conditions.

We start by exploring numerically the $\alpha$-acceleration technique
presented in section~\ref{alpha_acceleration}. We show in
Figure~\ref{Fig1Ex2}
\begin{figure}[t]
  \centering
  \includegraphics[scale=0.8]{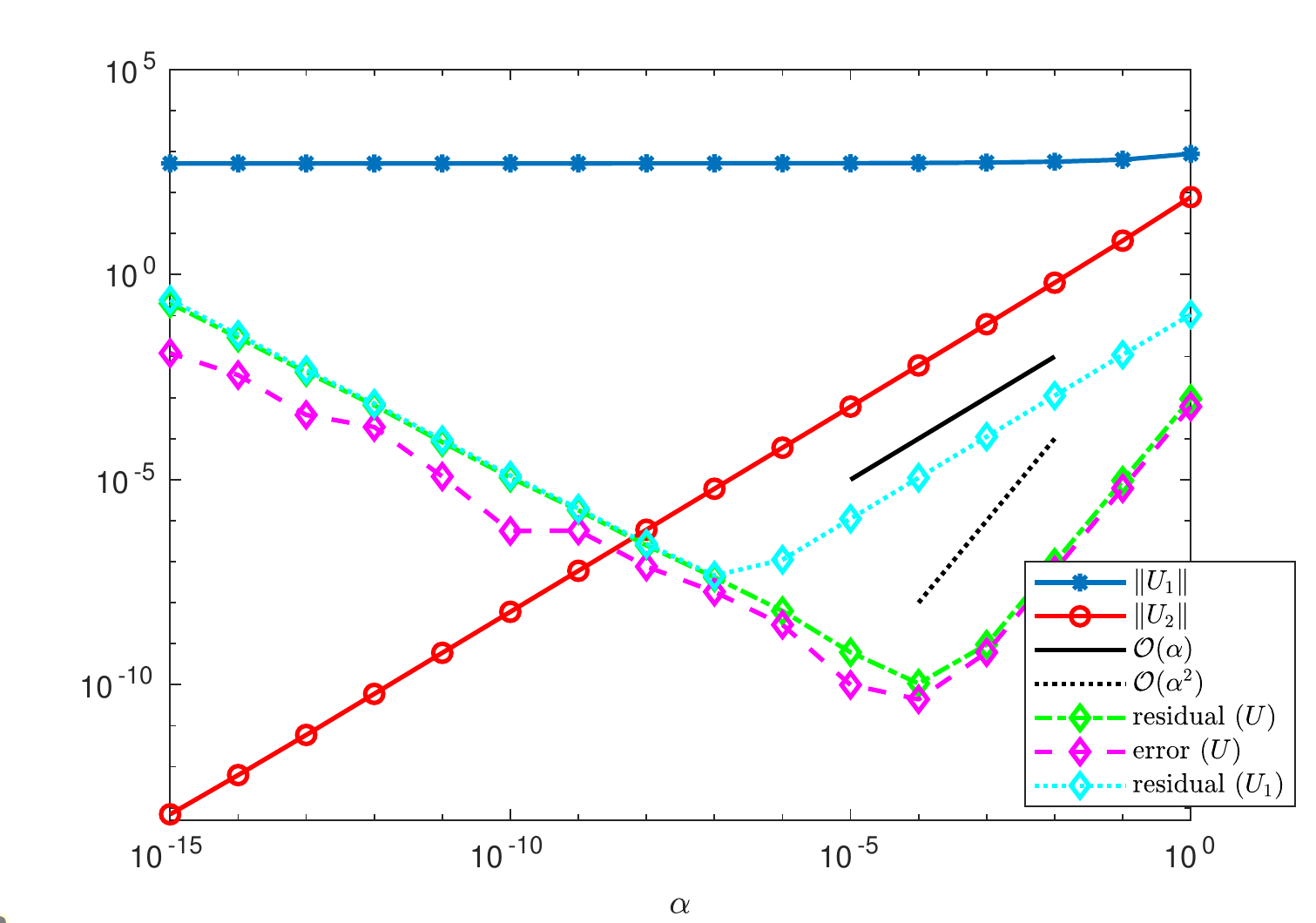}
  \caption{Advection-diffusion equation: $\|U_1\|_F$ (solid line with
    stars), $\|U_2\|_F$ (solid line with circles), relative
    residual norm achieved by $U$ (dashed-dotted line) and
    $U_1$ (dotted line), and the error $\|\widehat
      U-U\|_F/\|\widehat{U}\|_F$ (dashed line) for
    different values of $\alpha$, $\bar n=65\,536$, $\nu=2^{-5}$, and
    $\ell=64$. $\widehat{U}$ denotes the solution obtained by
    running GMRES on the all-at-once system with a small residual
    tolerance ($10^{-13}$).  }\label{Fig1Ex2}
\end{figure}
the results obtained by running a single
iteration of the FOM-like method, namely $\text{maxit}=1$ in
Algorithm~\ref{alg:galerkin}, for different values of $\alpha$. We set
$\bar n=65\,536$, $\nu=2^{-5}$, and use $\ell=64$ time steps, and
recall that we write the computed solution $U$ as
$U=U_1+U_2$ where $U_1$ and
$U_2$ are such that
$$
\begin{array}{l}
\text{vec}(U_1)=(D_\alpha^{-1}{F}^{-1}\otimes I)\text{vec}(L),\quad \text{vec}(L)= P^{-1}\text{vec}([\mathbf{u_0}+\tau\mathbf{f}_1,\ldots,\tau\mathbf{f}_\ell]D_\alpha{F}^T),\\
\\
\text{vec}(U_2)=-\alpha^{1/\ell} (D_\alpha^{-1}{F}^{-1}\otimes I)P^{-1}M(I+\alpha^{1/\ell} N^TP^{-1}M)^{-1}N^T\text{vec}(L).
\end{array}
$$
In Figure~\ref{Fig1Ex2} we show $\|U_1\|_F$ (solid line with
stars), $\|U_2\|_F$ (solid line with circles), the relative
residual norm achieved by the whole $U$ (dashed-dotted line)
and the one attained by using only $U_1$ (dotted line)
for different values of $\alpha$. We also
computed the \emph{exact} algebraic solution $\widehat{U}$ by
running GMRES on the all-at-once system with the very small tolerance
$10^{-13}$ on the relative residual norm. In Figure~\ref{Fig1Ex2} we
thus plot also the trend of $\|\widehat
  U-U\|_F/\|\widehat{U}\|_F$ (dashed line).

The first thing to notice from Figure~\ref{Fig1Ex2} is that a change
in $\alpha$ does not really have an impact on $\|U_1\|_F$;
this value remains (almost) constant for all the $\alpha$'s we tested.  On the other
hand, the trend of $\|U_2\|_F$ closely follows the values of
$\alpha$ showing how the contribution of $U_2$
strongly depends on the selected $\alpha$. If a $\mathcal{O}(10^{-7})$
relative residual norm is good enough\footnote{Notice that having this
  kind of accuracy in the algebraic problem is often exceeding the discretization error by orders of magnitude.}, $U_2$
can be completely neglected if $\alpha=\mathcal{O}(10^{-7})$ in this
example. This means that $U=U_1$ can be computed by
performing a single parallel-in-time loop. {\color{black} On the other hand, $U_2$ has an important role in the accuracy of the overall solution. Indeed, it is interesting to note that the error and the residual norm achieved by $U=U_1+U_2$
scale like $\alpha^2$ whereas the residual norm attained by $U_1$ depends only linearly on $\alpha$.}

Note also that $\alpha$ plays a role in the inner linear system we
need to solve to compute $U_2$. Figure~\ref{Fig1Ex2} shows
that the smaller $\alpha$, the more accurate a solution we get by
performing a single iteration of the FOM-like method. This is due to
the fact that the coefficient matrix in~\eqref{eq:innersystem_alpha}
becomes a small perturbation of the identity for small values of
$\alpha$.

As already mentioned, however, using too small values of $\alpha$
leads to a remarkable increase in the condition number of the
transformation matrix ${F}D_\alpha$. This is clearly visible in
Figure~\ref{Fig1Ex2}: for $\alpha\leq 10^{-4}$ the error starts
increasing, since $\kappa({F}D_\alpha)$ is becoming the dominant
factor polluting the quality of the computed solution
$U$.

\subsection{Comparison with other ParaDiag Methods}
  
We now compare our new ParaDiag scheme, namely Algorithm~\ref{alg:alpha_paradiag}, with
different state-of-the-art ParaDiag solvers. As already mentioned, in what follows we
report the number of parallel-in-time loops {\color{black}(\#PinT)} needed by the routines we test as a measure to
assess their computational cost, instead of
running time, since the latter strongly depends on the precise
implementations and computing infrastructures used. Recording fair
comparisons in terms of running times would not be straightforward.

The first ParaDiag technique we compare to is the one proposed
in~\cite{McDonald2018}, where preconditioned GMRES is used to
iteratively solve the $\bar n\ell\times \bar n\ell$ linear system
corresponding to the Kronecker form of~\eqref{eq:discrete_prob}. {\color{black}
The preconditioning operator we \BL{use} within GMRES with
right preconditioning\footnote{To have
  fair comparisons, we use a matrix oriented GMRES implementation, so
  that the Kronecker form of~\eqref{eq:discrete_prob} is never
  explicitly computed. See, e.g., \cite{PalittaK21}.
} \BL{is}
$$\mathcal{P}:x\mapsto (F\otimes I_{\bar n})(I_\ell\otimes (I_{\bar n}+\tau K)+\Pi_1\otimes I_{\bar n})(F^{-1}\otimes I_{\bar n})x.
$$
Therefore, at each iteration we need to apply $\mathcal{P}^{-1}$ to the current basis vector, namely} we need to perform a
parallel-in-time loop every time we apply the preconditioner. This
means that, in total, this routine performs $p+1$ parallel-in-time loops if $p$
denotes the number of GMRES iterations needed to converge. The
threshold on the relative residual norm we used in GMRES is
$\epsilon=10^{-8}$. {\color{black} Notice that the use of right preconditioning allows us to have access to the actual, unpreconditioned residual norm computed by GMRES. The latter quantity is thus comparable with the residual norm achieved by Algorithm~\ref{alg:alpha_paradiag}.}

The second ParaDiag technique we compare to is the new interpolation
scheme presented in~\cite[Section 3]{Kressner2022} and denoted by {\tt
  Ev-Int} in what follows. In this method, given two parameters
$\rho$ and $r$, one needs to perform $r$ parallel-in-time loops
involving different coefficient matrices. The quality of the computed
solution depends on $\rho$ and $r$. The authors in~\cite{Kressner2022}
do not comment much on the selection of $\rho$ and $r$. They
suggest to use $\rho=5\cdot10^{-4}$ and $r=2$, values
adopted in most of the experiments shown in~\cite{Kressner2022}.

In Table~\ref{tab2}
\begin{table}
 \centering
 \addtolength{\tabcolsep}{-2.8pt}
  \begin{tabular}{r r r| rr|rr|rr}
& & & \multicolumn{2}{c|}{Algorithm~\ref{alg:alpha_paradiag} ($\alpha=10^{-4}$)}& \multicolumn{2}{c|}{GMRES} & \multicolumn{2}{c}{{\tt Ev-Int}} \\
  $\bar n$  & $\ell$ &$\nu$ & {\color{black}\#PinT} & Rel. Res. & {\color{black}\#PinT} & Rel. Res. &{\color{black}\#PinT} & Rel. Res. \\
\hline
 \multirow{9}{*}{$16\,384$} & \multirow{3}{*}{32} & $10^{-1}$ & 3 & 8.41e-11 & 5 & 8.02e-13 & {\color{black}2}&7.26e-11 \\
 
 && $10^{-2}$ & 3 & 6.98e-12 & 5 & 7.28e-12 & {\color{black}2}&1.24e-11 \\

 && $10^{-3}$ & 3 & 2.77e-12 & 5 & 2.13e-10 & {\color{black}2}&7.22e-11 \\
 
  &\multirow{3}{*}{64}& $10^{-1}$ & 3 & 1.74e-11 & 5 & 1.12e-13 & {\color{black}2}&3.72e-11 \\

  && $10^{-2}$ & 3 & 1.42e-12 & 5 & 1.31e-12 & {\color{black}2}&4.19e-12 \\

  && $10^{-3}$ & 3 & 7.41e-13 & 5 & 9.42e-12 & {\color{black}2}&1.99e-11 \\

 & \multirow{3}{*}{128} & $10^{-1}$ & 3 & 1.20e-11 & 5 & 4.51e-14 & {\color{black}2}&2.30e-11 \\
    
  && $10^{-2}$ & 3 & 1.01e-12 & 5 & 4.24e-13 & {\color{black}2}&2.23e-12 \\
 
  && $10^{-3}$ & 3 & 3.99e-13 & 5 & 1.88e-12 & {\color{black}2}&8.11e-12 \\

    \hline
    \multirow{9}{*}{$65\,536$} & \multirow{3}{*}{32}& $10^{-1}$ & 3 & 3.42e-10 & 5 & 6.67e-13 & {\color{black}2}&8.58e-11 \\

    && $10^{-2}$ & 3 & 2.72e-11& 5 & 3.91e-12 & {\color{black}2}&1.09e-11 \\
    
    && $10^{-3}$ & 3 & 3.61e-12 & 5 & 3.37e-11 & {\color{black}2}&4.49e-11 \\

 & \multirow{3}{*}{64} & $10^{-1}$ & 3 & 7.26e-11 & 5 & 1.31e-13 & {\color{black}2}&3.62e-11 \\

 && $10^{-2}$ & 3 & 5.21e-12 & 5 & 6.73e-13 & {\color{black}2}&3.33e-12 \\

 && $10^{-3}$ & 3 & 7.29e-13 & 5 & 8.71e-13 & {\color{black}2}&1.00e-11 \\

  & \multirow{3}{*}{128} & $10^{-1}$ & 3 & 4.84e-11 & 5 & 1.26e-13 &{\color{black}2}& 3.07e-11 \\

  & & $10^{-2}$ & 3 & 3.71e-12 & 5 & 2.11e-13 &{\color{black}2}& 2.28e-12 \\

  & & $10^{-3}$ & 3 & 5.40e-13 & 5 & 1.50e-13 & {\color{black}2}&3.55e-12 \\

  \end{tabular}
 \caption{Advection-Diffusion equation: results for different values of $\bar
   n$, $\ell$, and $\nu$.\label{tab2}}
\end{table}
we show the results for different values of $\bar n$, $\ell$, and the
viscosity parameter $\nu$. We use Algorithm~\ref{alg:alpha_paradiag}
with $\alpha=10^{-4}$, and recall that this algorithm performs $m+2$
parallel-in-time loops where $m$ is the number of iterations needed
by Algorithm~\ref{alg:galerkin} to converge. From the results in
Table~\ref{tab2} we see that the number of parallel-in-time loops performed
by both Algorithm~\ref{alg:alpha_paradiag} and GMRES are very robust
with respect to $\bar n$, $\ell$, and $\nu$. We also see that the
accuracy attained by our solver improves by decreasing $\nu$, for
fixed $\bar n$ and $\ell$, whereas GMRES shows the opposite trend, in
general. {\tt Ev-Int} has a less regular behavior in this regard.

Notice however that, even though Algorithm~\ref{alg:alpha_paradiag} is
often more accurate than {\tt Ev-Int}, the latter algorithm only performs two
parallel-in-time loops, in contrast to the three loops performed by our new
ParaDiag algorithm. By setting $r=3$, similar results in terms of
accuracy can be obtained for {\tt Ev-Int} as well.

To make a fair
comparison between our procedure and {\tt Ev-Int}, in
Table~\ref{tab3} we show
\begin{table}
 \setlength{\tabcolsep}{.4em}
 \centering
\begin{tabular}{r r r| cc|cc}
& & & \multicolumn{2}{c|}{{\color{black}\#PinT$=2$}}& \multicolumn{2}{c}{{\color{black}\#PinT$=1$}}  \\
  $\bar n$  & $\ell$ &$\nu$ & Alg.~\ref{alg:alpha_paradiag} ($\alpha=10^{-4}$) & {\tt Ev-Int} & Alg.~\ref{alg:alpha_paradiag} ($\alpha=10^{-6}$) & {\tt Ev-Int} \\
\hline
 \multirow{9}{*}{$16\,384$} & \multirow{3}{*}{32} & $10^{-1}$ & 8.44e-11 & 7.26e-11 & 1.87e-8 & 8.95e-6 \\

 && $10^{-2}$ & 6.97e-12 & 1.24e-11 & 3.97e-8 & 1.98e-5  \\

 && $10^{-3}$ & 3.14e-12 & 7.22e-11 & 5.88e-8 & 2.94e-5  \\

  &\multirow{3}{*}{64}& $10^{-1}$ & 1.76e-11 & 3.72e-11 & 1.09e-8 & 5.45e-6 \\

  && $10^{-2}$ & 1.43e-12 & 4.19e-12 & 2.45e-8 & 1.22e-5  \\

  && $10^{-3}$ & 8.43e-13 & 1.99e-11 & 3.99e-8 & 1.99e-5  \\

 & \multirow{3}{*}{128} & $10^{-1}$ & 1.20e-11 & 2.30e-11 & 8.23e-9 & 3.55e-6  \\

  && $10^{-2}$ & 1.01e-12 & 2.23e-12 & 1.59e-8 & 7.96e-6  \\

  && $10^{-3}$ & 4.30e-13 & 8.11e-12 & 2.79e-8 & 1.39e-5  \\

    \hline
    \multirow{9}{*}{$65\,536$} & \multirow{3}{*}{32}& $10^{-1}$ & 3.43e-10 & 8.58e-11 & 2.84e-8 & 8.66e-6 \\

    && $10^{-2}$ & 2.71e-11 & 1.09e-11& 3.84e-8 & 1.92e-5  \\

    && $10^{-3}$ & 3.73e-12 & 4.49e-11 & 5.64e-8 & 2.82e-5 \\

 & \multirow{3}{*}{64} & $10^{-1}$ & 7.27e-11 & 3.62e-11 & 1.17e-8 & 5.26e-6 \\

 && $10^{-2}$ & 5.21e-12 & 3.33e-12 & 2.35e-8 & 1.17e-5  \\

 && $10^{-3}$ & 7.57e-13 & 1.00e-11 & 3.81e-8 & 1.90e-5  \\

  & \multirow{3}{*}{128} & $10^{-1}$ & 4.83e-11 & 3.07e-11 & 1.78e-8 & 3.43e-6  \\

  & & $10^{-2}$ & 3.71e-12 & 2.28e-12 & 1.52e-8 & 7.59e-6  \\

  & & $10^{-3}$ & 5.46e-13 & 3.55e-12 & 2.66e-8 & 1.33e-5  \\

\end{tabular}
 \caption{Advection-Diffusion equation: results for different values of $\bar
   n$, $\ell$, and $\nu$.\label{tab3}}
\end{table}
the relative residual norms obtained by fixing the main computational
cost of the two algorithms, namely we perform the same number of
parallel-in-time loops in both schemes. For {\tt Ev-Int} we use $r=2$
(same results as in Table~\ref{tab2}) and $r=1$. For
Algorithm~\ref{alg:alpha_paradiag} ($\alpha=10^{-4}$) we perform only
two parallel-in-time loops by approximating the matrix
in~\eqref{eq:innersystem_alpha} by the identity, namely the FOM-like
method is not performed and we set $\mathbf{x}_m=\mathbf{b}$ in
line~\ref{alg:linegalerkin} of Algorithm~\ref{alg:alpha_paradiag}. The
single parallel-in-time loop scenario is addressed by reporting the
accuracy attained by $U_1$ in
Algorithm~\ref{alg:alpha_paradiag} for $\alpha=10^{-6}$.

As shown in Table~\ref{tab3}, our new ParaDiag algorithm is in general
at least as accurate as {\tt Ev-Int} in case of two parallel-in-time
loops. On the other hand, whenever a single parallel-in-time loop is
performed, our new Algorithm~\ref{alg:alpha_paradiag} ($\alpha=10^{-6}$) is
often three orders of magnitude more accurate than {\tt
  Ev-Int}. Therefore, in general, our new ParaDiag algorithm achieves
better results in terms of accuracy than {\tt Ev-Int} whenever a cap
on the computational cost of the adopted solver is imposed. {\color{black} We would like to mention, however, that a careful tuning of the {\tt Ev-Int} parameters may improve the performance of the solver.}


\section{Conclusions}\label{Conclusions}

We presented a new ParaDiag algorithm which fully exploits the circulant-plus-low-rank structure of the discrete operator stemming from {\color{black}the discretization of evolution problems~\eqref{eq:diff_prob} by} BDFs, one of the most commonly used family of implicit time integrators.
A clever use of the matrix-oriented
Sherman-Morrison-Woodbury formula along with the design of an ad-hoc projection scheme make our new strategy very successful.

We studied our algorithm for parabolic problems, and
also introduced a variant using $\alpha-$acceleration, based on
$\alpha-$circulant matrices. A comparison with recent ParaDiag
techniques from the literature shows that our new ParaDiag algorithm
is competitive and capable of delivering superior accuracy for
comparable cost.

Our methodology can be easily generalized to other discretization schemes as long as the discrete time operator can be written as a circulant matrix plus a low-rank correction. {\color{black} For instance, in~\cite[Section 4]{Kressner2022} it is shown that the all-at-once discretization of~\eqref{eq:diff_prob} by Runge-Kutta methods leads to discrete problems similar to~\eqref{eq:discrete_prob}, with a comparable $\Sigma_1$. Therefore, our \BL{ParaDiag scheme can also be used to parallelize these Runge-Kutta methods in time.}}

\section*{Acknowledgments}

This work was supported by the Swiss National Science Foundation, and
part has been carried out at CIRM in Marseille, France, in the context
of the Morlet chair of the first author. The second author is member
of the Italian INdAM Research group GNCS. His work was partially supported by the research project ``Tecniche avanzate per problemi evolutivi: discretizzazione, algebra lineare numerica, ottimizzazione'' (INdAM
- GNCS Project CUP\_E55F22000270001).

{\color{black} We would like to thank the anonymous reviewers for their valuable comments and remarks.}

{\color{black}
\section*{Appendix}

 \BL{We show now} the proof of Theorem~\ref{th:alg_galerkin}.

\begin{proof}
If $K$ is symmetric, then by its definition also $J_\ell$ is
symmetric.  The only complex values involved in the definition of
$J_\ell$ are the $\pi_i$'s and $\gamma_i$'s and they appear as either
diagonal elements (the former ones) or scalar multipliers (the latter
ones).  The eigenvalues $\pi_i$ of the circulant matrix $C_1$ are of
the form $\pi_i=\omega^{-(\ell-1)(i-1)}$ for all $i=1,\ldots,\ell$,
where $\omega=e^{-\frac{2\pi \imath}{\ell}}\in\mathbb{C}$. Therefore,
$|\pi_i|\leq 1$ for all $i=1,\ldots,\ell$ and they come in complex
conjugate pairs. Moreover,
$$
  \begin{bmatrix}
    \gamma_1\\
    \vdots\\
    \gamma_\ell\\
   \end{bmatrix}={F}^{-T}\mathbf{e}_\ell=\frac{1}{\ell} \mathcal{\bar F}\mathbf{e}_\ell=\frac{1}{\ell}\begin{bmatrix}
    1\\
    \omega^{-(\ell-1)}\\
    \omega^{-2(\ell-1)}\\
    \vdots\\
    \omega^{-(\ell-1)^2}\\
    \end{bmatrix}=\frac{1}{\ell}\begin{bmatrix}
    1\\
    \pi_1\\
    \pi_2\\
    \vdots\\
    \pi_\ell\\
    \end{bmatrix},
$$
so that also the $\gamma_i$'s come in complex conjugate pairs, and
$|\gamma_i|\leq 1/\ell\leq 1$ for any $i=1,\ldots,\ell$, $\ell\geq 1$.

Since $J_\ell$ is symmetric, $|\pi_i|\leq 1$ for any $i$, and the $\pi_i$'s and $\gamma_i$'s all
come in complex conjugate pairs, $J_\ell$ is Hermitian.

We now show that it is also positive definite. For any
$z\in\mathbb{C}^\ell$ satisfying $\|z\|=1$, we have $
\text{Re}\left(z^*((1-\pi_i)I_{\bar n}+\tau K)z\right)>0 $ since $K$
is SPD, $\tau>0$ and $|\pi_i|\leq1$ for any $i=1,\ldots,\ell$. This
means that also $((1-\pi_i)I_{\bar n}+\tau K)^{-1}$ is positive
definite for any $i$. Furthermore, recalling that $\gamma_i=\pi_i/\ell$, we have
\begin{align}\label{eq:realpart_pi}
 \min_{z\in\mathbb{C}^\ell,\|z\|=1}\text{Re}(z^*J_\ell z)=&1+\sum_{i=1}^\ell\text{Re}(\gamma_i)\min_{z\in\mathbb{C}^\ell,\|z\|=1}\text{Re}(z^*((1-\pi_i)I+\tau K)^{-1}z)\notag\\
=&1+\frac{1}{\ell}\sum_{i=1}^\ell\frac{\text{Re}(\pi_i)}{1-\text{Re}(\pi_i)+\tau\lambda_{\max}(K)}.
\end{align}
By construction we know that $\sum_{i=1}^\ell\text{Re}(\pi_i)=0$. Moreover, if we denote
by $\mathfrak{J}$, $\mathfrak{K}$ the index sets such that $\text{Re}(\pi_i)\geq0$ and $\text{Re}(\pi_i)<0$, respectively, then, for an even $\ell$, $\mathfrak{J}$ and $\mathfrak{K}$ have the same cardinality. In particular, $\#\mathfrak{J}=\#\mathfrak{K}=\ell/2$ and we have
$$0=\sum_{i=1}^\ell\text{Re}(\pi_i)=\sum_{j\in\mathfrak{J}}\text{Re}(\pi_j)+\sum_{k\in\mathfrak{K}}\text{Re}(\pi_k)=\sum_{j\in\mathfrak{J}}(\text{Re}(\pi_j)-\text{Re}(\pi_j)),
$$
namely for a given $j\in\mathfrak{J}$ there exists an index $k\in\mathfrak{K}$ such that $\text{Re}(\pi_j)=-\text{Re}(\pi_k)$ if $\ell$ is even.

On the other hand, if $\ell$ is odd, we have $\#\mathfrak{J}=(\ell-1)/2+1$, $\#\mathfrak{K}=(\ell-1)/2$ and we can write
$$0=\sum_{i=1}^\ell\text{Re}(\pi_i)=\sum_{j\in\mathfrak{J}}\text{Re}(\pi_j)+\sum_{k\in\mathfrak{K}}\text{Re}(\pi_k)=
1+\sum_{j\in\mathfrak{J}, \pi_j\neq 1}\text{Re}(\pi_j)+\sum_{k\in\mathfrak{K}}\text{Re}(\pi_k),
$$
which means that
$$0<\sum_{j\in\mathfrak{J}, \pi_j\neq 1}\text{Re}(\pi_j)=-1-\sum_{k\in\mathfrak{K}}\text{Re}(\pi_k).
$$

However, in~\eqref{eq:realpart_pi} we have a
weighted sum of the real parts of the eigenvalues $\pi_i$. In particular, we can write

$$
\frac{1}{\ell}\sum_{i=1}^\ell\frac{\text{Re}(\pi_i)}{1-\text{Re}(\pi_i)+\tau\lambda_{\max}(K)}=\sum_{i=1}^\ell\text{Re}(\pi_i)w(\pi_i),$$
where
$$ w(\pi_i):=\frac{1}{\ell(1-\text{Re}(\pi_i)+\tau\lambda_{\max}(K))}> 0,\quad \text{for any } i=1,\ldots,\ell.
$$
We now focus on the scalars $w(\pi_i)$  and we define
$$w(\pi_i):=\left\{\begin{array}{l}
\check{w}(\pi_i), \quad\text{if }\;i\in\mathfrak{K}\BL{,}\\
\hat{w}(\pi_i), \quad\text{if }\;i\in\mathfrak{J}.
\end{array}\right.
$$
A direct computation shows that $\hat{w}(\pi_j)\geq \check{w}(\pi_k)$ for any $j\in\mathfrak{J},k\in\mathfrak{K}$, since
\begin{equation}\label{eq:minmax}
\min_{j\in\mathfrak{J}}\hat{w}(\pi_j)\geq
\frac{1}{\ell\tau\lambda_{\max}(K)}\geq\max_{k\in\mathfrak{K}}\check{w}(\pi_k).
\end{equation}
For an even $\ell$, it holds that
\begin{align*}
\sum_{i=1}^\ell\text{Re}(\pi_i)w(\pi_i)=&\sum_{j\in\mathfrak{J}}\text{Re}(\pi_j)\hat{w}(\pi_j)+\sum_{k\in\mathfrak{K}}\text{Re}(\pi_k)\check{w}(\pi_k)\\
=&\sum_{j\in\mathfrak{J}}\text{Re}(\pi_j)(\hat{w}(\pi_j)-\check{w}(\pi_{k_j}))
\geq0,
\end{align*}
where the index $k_j\in\mathfrak{K}$ is such that
$\text{Re}(\pi_{k_j})=-\text{Re}(\pi_{j})$ for a given $j\in\mathfrak{J}$.

On the other hand, if $\ell$ is odd and $j^*\in\mathfrak{J}$ is such that $\pi_{j^*}=1$, we can write
\begin{align*}
 \sum_{i=1}^\ell\text{Re}(\pi_i)w(\pi_i)=&\,
\hat{w}(\pi_{j^*})+\sum_{j\in\mathfrak{J}\setminus\{j^*\}}\text{Re}(\pi_j)\hat{w}(\pi_j)+\sum_{k\in\mathfrak{K}}\text{Re}(\pi_k)\check{w}(\pi_k)\\
\geq &\,\hat{w}(\pi_{j^*})+\frac{1}{\ell\tau\lambda_{\max}(K)}\sum_{j\in\mathfrak{J}\setminus\{j^*\}}\text{Re}(\pi_j)+\sum_{k\in\mathfrak{K}}\text{Re}(\pi_k)\check{w}(\pi_k)\\
=&\,\hat{w}(\pi_{j^*})-\frac{1}{\ell\tau\lambda_{\max}(K)}-\frac{1}{\ell\tau\lambda_{\max}(K)}\sum_{k\in\mathfrak{K}}\text{Re}(\pi_k)+\sum_{k\in\mathfrak{K}}\text{Re}(\pi_k)\check{w}(\pi_k)\\
=&\,\hat{w}(\pi_{j^*})-\frac{1}{\ell\tau\lambda_{\max}(K)}+\sum_{k\in\mathfrak{K}}\text{Re}(\pi_k)\left(\check{w}(\pi_k)-\frac{1}{\ell\tau\lambda_{\max}(K)}\right)\geq 0.
%
\end{align*}
The nonegativity of the quantity above follows from~\eqref{eq:minmax}. Indeed, $\hat{w}(\pi_{j^*})\geq\frac{1}{\ell\tau\lambda_{\max}(K)}$ whereas $\check{w}(\pi_k)-\frac{1}{\ell\tau\lambda_{\max}(K)}\leq 0$ for any $k\in\mathfrak{K}$. On the other hand, $\text{Re}(\pi_k)<0$ for any $k\in\mathfrak{K}$ so that
$\text{Re}(\pi_k)\left(\check{w}(\pi_k)-\frac{1}{\ell\tau\lambda_{\max}(K)}\right)\geq 0$ for any $k\in\mathfrak{K}$.

Therefore, for any $\ell$, we have
$$1+\frac{1}{\ell}\sum_{i=1}^\ell\frac{\text{Re}(\pi_i)}{1-\text{Re}(\pi_i)+\tau\lambda_{\max}(K)}>0,$$
and this shows the positive definiteness of $J_\ell$.

To conclude, we show the upper bound~\eqref{eq:ThAlg}.
For any $\ell$, we have
\begin{align*}
 \kappa(J_\ell)=&\,\frac{\max_i|\lambda_i(J_\ell)|}{\min_i|\lambda_i(J_\ell)|}=
 \frac{\max_{z\in\mathbb{C}^\ell,\|z\|=1}|z^*J_\ell z|}{\min_{z\in\mathbb{C}^\ell,\|z\|=1}|z^*J_\ell z|}\\
 =&\,\frac{1+\max_{z\in\mathbb{C}^\ell,\|z\|=1}|\sum_{i=1}^\ell \gamma_i
z^*((1-\pi_i)I_{\bar n}+\tau K)^{-1}z|}{1+\min_{z\in\mathbb{C}^\ell,\|z\|=1}|\sum_{i=1}^\ell \gamma_i
z^*((1-\pi_i)I_{\bar n}+\tau K)^{-1}z|}.
\end{align*}
Recalling that $|\gamma_i|\leq 1/\ell$ and $0\leq|1-\pi_i|\leq2$ for any $i$ we get
\begin{align*}
\kappa(J_\ell)
 \leq&\,\frac{1+\max_i\max_{z\in\mathbb{C}^\ell,\|z\|=1}|
z^*((1-\pi_i)I_{\bar n}+\tau K)^{-1}z|}{1+\ell\min_i\min_{z\in\mathbb{C}^\ell,\|z\|=1}|z^*((1-\pi_i)I_{\bar n}+\tau K)^{-1}z|}\\
= &\,\frac{1+\max_i\frac{1}{\min_{z\in\mathbb{C}^\ell,\|z\|=1}|
z^*((1-\pi_i)I_{\bar n}+\tau K)z|}}{1+\ell\min_i\frac{1}{\max_{z\in\mathbb{C}^\ell,\|z\|=1}|z^*((1-\pi_i)I_{\bar n}+\tau K)z|}} \\
\leq&\,\frac{1+\frac{1}{\tau \lambda_{\min}(K)}}{1+\frac{\ell}{2+\tau\lambda_{\max}(K)}}=\frac{\tau\lambda_{\min}(K)+1}{2+\tau\lambda_{\max}(K)+\ell}\cdot\frac{2+\tau\lambda_{\max}(K)}{\tau\lambda_{\min}(K)}\\
=&\,\frac{\tau\lambda_{\min}(K)}{\tau\lambda_{\max}(K)}\frac{1+\frac{1}{\tau\lambda_{\min}(K)}}{1+\frac{2+\ell}{\tau\lambda_{\max}(K)}}\cdot\frac{\tau\lambda_{\max}(K)}{\tau\lambda_{\min}(K)}\left(1+\frac{2}{\tau\lambda_{\max}(K)}\right)\\
\leq&\, 1+\frac{1}{\tau\lambda_{\min}(K)
}.
\\
\end{align*}
\end{proof}
}
\bibliography{paradiag}

\end{document}